\documentclass[reqno,11pt]{amsart}
\usepackage{amsmath}
\usepackage{amsfonts}
\usepackage{amssymb}
\usepackage{stmaryrd}
\usepackage{graphicx}
\usepackage{color}
\usepackage{cite}
\usepackage{times}
\usepackage{tikz}
\usepackage{pgfplots}
\pgfplotsset{compat=1.10}
\usepgfplotslibrary{fillbetween}
\usetikzlibrary{patterns}

\newtheorem{theorem}{Theorem}[section]
\newtheorem{corollary}[theorem]{Corollary}
\newtheorem{lemma}[theorem]{Lemma}

\newtheorem{definition}[theorem]{Definition}
\newtheorem{remark}[theorem]{Remark}
\newtheorem{example}[theorem]{Example}

\DeclareMathAlphabet{\mathpzc}{OT1}{pzc}{m}{it}

\newcommand{\rd}{\mathrm{d}}
\definecolor{cadmiumgreen}{rgb}{0.0, 0.42, 0.24}
\numberwithin{equation}{section}
\title[An evolution model for MEMS]{Variational solutions to an evolution model for MEMS\\ with heterogeneous dielectric properties}

\author{Philippe Lauren\c{c}ot}
\address{Institut de Math\'ematiques de Toulouse, UMR~5219, Universit\'e de Toulouse, CNRS \\ F--31062 Toulouse Cedex 9, France}
\email{laurenco@math.univ-toulouse.fr}
\thanks{Partially supported by the CNRS Projet International de Coop\'eration Scientifique PICS07710}

\author{Christoph Walker}
\address{Leibniz Universit\"at Hannover\\ Institut f\" ur Angewandte Mathematik \\ Welfengarten 1 \\ D--30167 Hannover\\ Germany}
\email{walker@ifam.uni-hannover.de}

\begin{document}

\date{\today}

\begin{abstract}
The existence of weak solutions to the obstacle problem for a nonlocal semilinear fourth-order parabolic equation is shown, using its underlying gradient flow structure. The model governs the dynamics of a microelectromechanical system with heterogeneous dielectric properties.
\end{abstract}


\maketitle

\section{Introduction}

The existence of variational solutions is shown for an evolution problem describing the space-time dynamics of a microelectromechanical system (MEMS) with heterogeneous dielectric properties. Specifically, a MEMS device such a switch is made of a thin rigid conducting plate above which a thin conducting elastic plate is suspended and clamped at its boundary. The shape of the undeformed elastic plate is identical to that of of the rigid one. Holding the two plates at different electrostatic potentials generates a deformation of the top elastic plate to compensate the induced electrostatic force. It is by now well-known that a sufficiently large potential difference can lead to a \textit{pull-in instability} or \textit{touchdown}, a situation which corresponds to the top plate coming into contact with the bottom one and results in a short circuit due to the potential difference \cite{AmEtal,BG01,LWBible,Pel01a,PeB03}. Clearly, such a phenomenon may alter the properties or the operating conditions of the MEMS device. However, it can be prevented, for instance, by covering the ground plate with an insulating layer of positive thickness \cite{AmEtal,BG01,LLG14,LLG15}. We consider this situation herein and thus assume that the bottom plate is covered by a non-deformable layer of positive thickness, possibly having heterogeneous dielectric properties characterized by a permittivity $\sigma_1>0$, which differs from the constant permittivity $\sigma_2>0$ of the surrounding medium. Touchdown may still occur, in the sense that the top plate may come into contact with the upper side of the insulating layer. However, such a situation does not generate a singularity, as the top plate cannot penetrate the layer. Assuming further that the physical state of the MEMS device is fully described by the deformation $u$ of the top plate and the electrostatic potential $\psi_u$ between the two plates, the dynamics of the MEMS is then governed by the competition between mechanical and electrostatic forces, and is given by a time relaxation towards critical points of the total energy, the latter including mechanical, contact, and electrostatic contributions. 

To convert this rough description of the model into mathematical equations, we assume that there is no variation in one of the two horizontal directions and consider a two-dimensional MEMS device in which the rigid ground plate and the undeformed elastic plate have the same one-dimensional shape $D:=(-L,L)$, $L>0$, the former being located at height $z=-H-d$, $H>0$, $d>0$, while the latter is clamped at its boundary $(\pm L,0)$. The bottom plate is then $D\times \{-H-d\}$ and it is covered by an insulating layer 
\begin{equation*}
\Omega_1 := D\times (-H-d,-H)
\end{equation*} 
of positive thickness $d$. The dielectric properties of $\Omega_1$ are characterized by the permittivity $\sigma_1>0$ which may vary with the horizontal coordinate $x\in D$ and the height $z\in (-H-d,-H)$. As for the elastic top plate, we assume that the relevant physical framework is restricted to small deformations and, denoting the vertical deformation of the top plate by $u$, the top plate is given by
\begin{equation*}
\mathfrak{G}(u) := \{ (x,z)\in D\times\mathbb{R}\ :\ z=u(x)\}\,. 
\end{equation*}
Recalling that the top plate cannot deformed beyond the surface $\Sigma:= D\times \{-H\}$ of the insulated layer $\Omega_1$, the deformation $u$ is bounded from below by $-H$. The region $\Omega_2(u)$ between the top plate and the surface of the insulating layer is defined by
\begin{equation*}
\Omega_2(u) := \{ (x,z)\in D\times\mathbb{R}\ :\ -H < z < u(x)\}\,.
\end{equation*}
The dielectric permittivity is assumed to be a positive constant $\sigma_2>0$ in $\Omega_2(u)$ and differs in general from $\sigma_1(\cdot,-H)$, so that there is a jump discontinuity of the permittivity across the interface 
\begin{equation*}
\Sigma(u) := \{ (x,-H)\ : \ x\in D \;\text{ and }\; u(x)>-H\}\,,
\end{equation*}
separating $\Omega_1$ and $\Omega_2(u)$. Observe that $\Omega_2(u)$ is an open subset of $D\times (-H,\infty)$, which is connected when the coincidence set 
\begin{equation}
\mathcal{C}(u) := \{ x\in D\, :\, u(x)=-H\}\,. \label{coincidence}
\end{equation}
is empty, while it is disconnected and non-Lipschitz (if $u$ is smooth enough) otherwise. The two situations are depicted in Figure~\ref{Fig1} and Figure~\ref{Fig2}, respectively.  Independent of whether or not $\mathcal{C}(u)$ is empty, the domain
\begin{equation*}
\Omega(u) := \Omega_1 \cup \Sigma(u) \cup \Omega_2(u)\,, 
\end{equation*}
is Lipschitz (again if $u$ is smooth enough).

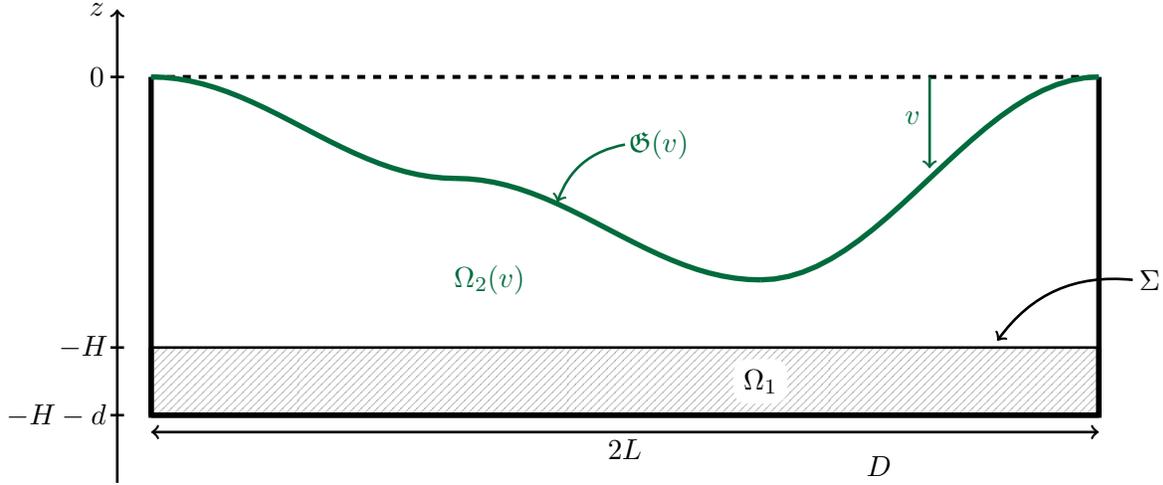
\begin{figure}
	\begin{tikzpicture}[scale=0.9]
	\draw[black, line width = 1.5pt, dashed] (-7,0)--(7,0);
	\draw[black, line width = 2pt] (-7,0)--(-7,-5);
	\draw[black, line width = 2pt] (7,-5)--(7,0);
	\draw[black, line width = 2pt] (-7,-5)--(7,-5);
	\draw[black, line width = 1pt] (-7,-4)--(7,-4);
	\draw[black, line width = 2pt, fill=gray, pattern = north east lines, fill opacity = 0.5] (-7,-4)--(-7,-5)--(7,-5)--(7,-4);
	\draw[cadmiumgreen, line width = 2pt] plot[domain=2:7] (\x,{-1.5-1.5*cos((pi*(\x-2)/5) r)});
	\draw[cadmiumgreen, line width = 2pt] plot[domain=-2.5:2] (\x,{-2.25-0.75*cos((pi*(2-\x)/4.5) r)});
	\draw[cadmiumgreen, line width = 2pt] plot[domain=-7:-2.5] (\x,{-0.75-0.75*cos((pi*(\x+2.5)/4.5) r)});
	\draw[cadmiumgreen, line width = 1pt, arrows=->] (4.5,0)--(4.5,-1.35);
	\node at (4.25,-0.6) {${\color{cadmiumgreen} v}$};
	\node[draw,rectangle,white,fill=white, rounded corners=5pt] at (2,-4.5) {$\Omega_1$};
	\node at (2,-4.5) {$\Omega_1$};
	\node at (-2,-3) {${\color{cadmiumgreen} \Omega_2(v)}$};
	\node at (3.75,-5.75) {$D$};
	\node at (7.75,-3) {$\Sigma$};
	\draw (7.5,-3) edge[->,bend right, line width = 1pt] (5.5,-3.9);
	\node at (-7.8,1) {$z$};
	\draw[black, line width = 1pt, arrows = ->] (-7.5,-6)--(-7.5,1);
	\node at (-8.4,-5) {$-H-d$};
	\draw[black, line width = 1pt] (-7.6,-5)--(-7.4,-5);
	\node at (-8,-4) {$-H$};
	\draw[black, line width = 1pt] (-7.6,-4)--(-7.4,-4);
	\node at (-7.8,0) {$0$};
	\draw[black, line width = 1pt] (-7.6,0)--(-7.4,0);
	\node at (0,-5.5) {$2L$};
	\draw[black, line width = 1pt, arrows = <->] (-7,-5.25)--(7,-5.25);
	\node at (0.5,-1) {${\color{cadmiumgreen} \mathfrak{G}(v)}$};
	\draw (0,-1) edge[->,bend right,cadmiumgreen, line width = 1pt] (-1,-1.85);
	\end{tikzpicture}
	\caption{Geometry of $\Omega(u)$ for a state $u=v$ with empty coincidence set (\textcolor{cadmiumgreen}{green}).}\label{Fig1}
\end{figure}

\begin{figure}
	\begin{tikzpicture}[scale=0.9]
	\draw[black, line width = 1.5pt, dashed] (-7,0)--(7,0);
	\draw[black, line width = 2pt] (-7,0)--(-7,-5);
	\draw[black, line width = 2pt] (7,-5)--(7,0);
	\draw[black, line width = 2pt] (-7,-5)--(7,-5);
	\draw[black, line width = 1pt] (-7,-4)--(7,-4);
	\draw[black, line width = 2pt, fill=gray, pattern = north east lines, fill opacity = 0.5] (-7,-4)--(-7,-5)--(7,-5)--(7,-4);
	\draw[blue, line width = 2pt] plot[domain=3:7] (\x,{-2-2*cos((pi*(3-\x)/4) r)});
	\draw[blue, line width = 2pt] (-1,-4)--(3,-4);
	\draw[blue, line width = 2pt] plot[domain=-7:-1] (\x,{-2-2*cos((pi*(1+\x)/6) r)});
	\draw[blue, line width = 1pt, arrows=->] (5,0)--(5,-1.85);
	\draw[blue, dashed, line width = 2pt] (-7,-4)--(-1,-4);
	\draw[blue, dashed, line width = 2pt] (3,-4)--(7,-4);
	\node at (4.7,-1) {${\color{blue} w}$};
	\node[draw,rectangle,white,fill=white, rounded corners=5pt] at (2,-4.5) {$\Omega_1$};
	\node at (2,-4.5) {$\Omega_1$};
	\node at (-5.75,-1.5) {${\color{blue} \Omega_2(w)}$};
	\node at (6,-2.5) {${\color{blue} \Omega_2(w)}$};
	\node at (3.75,-5.75) {$D$};
	\node at (8.05,-3) {${\color{blue} \Sigma(w)}$};
	\draw (7.5,-3) edge[->,bend right, blue, line width = 1pt] (5.5,-3.9);
	\node at (-4.5,-3) {${\color{blue} \Sigma(w)}$};
	\draw (-5.05,-3) edge[->,bend right, blue, line width = 1pt] (-6,-3.9);
	\node at (-7.8,1) {$z$};
	\draw[black, line width = 1pt, arrows = ->] (-7.5,-6)--(-7.5,1);
	\node at (-8.4,-5) {$-H-d$};
	\draw[black, line width = 1pt] (-7.6,-5)--(-7.4,-5);
	\node at (-8,-4) {$-H$};
	\draw[black, line width = 1pt] (-7.6,-4)--(-7.4,-4);
	\node at (-7.8,0) {$0$};
	\draw[black, line width = 1pt] (-7.6,0)--(-7.4,0);
	\node at (0,-5.5) {$2L$};
	\draw[black, line width = 1pt, arrows = <->] (-7,-5.25)--(7,-5.25);
	\node at (2,-3) {${\color{blue} \mathcal{C}(w)}$};
	\draw (1.45,-3) edge[->,bend right,blue, line width = 1pt] (0.5,-3.95);
	\node at (-2.6,-1) {${\color{blue} \mathfrak{G}(w)}$};
	\draw (-3.2,-1) edge[->,bend right,blue, line width = 1pt] (-4.0,-1.9);
	\end{tikzpicture}
	\caption{Geometry of $\Omega(u)$ for a state $u=w$ with non-empty coincidence set (\textcolor{blue}{blue}).}\label{Fig2}
\end{figure}

The total energy $E(u)$ of the MEMS device is given by 
\begin{equation}
E(u) = E_m(u) + E_c(u) + E_e(u)\,, \label{ent}
\end{equation}
where 
\begin{itemize}
	\item [--] the mechanical energy is
	\begin{equation*}
	E_m(u):=\frac{\beta}{2} \|\partial_x^2u\|_{2}^2 +\left( \frac{\tau}{2} + \frac{a}{4}\|\partial_x u\|_{2}^2 \right) \|\partial_x u\|_{2}^2\,,
	\end{equation*}
	including contributions from bending ($\beta>0$), stretching due to axial tension ($\tau>0$), and self-stretching due to elongation ($a>0$). Here, $\|\cdot \|_2$ denotes the norm in $L_2(D)$;
	\item [--] the contact energy is
	\begin{equation*}
	E_c(u) := \int_D \mathbb{I}_{[-H,\infty)}(u)\,\rd x\,,
	\end{equation*}
	where $\mathbb{I}_{[-H,\infty)}$ denotes the indicator function of the interval $[-H,\infty)$;
	\item [--] the electrostatic energy is
	\begin{equation*}
	E_e(u) := -\frac{1}{2} \int_{\Omega(u)} \sigma \vert\nabla \psi_u\vert^2\,\rd (x,z)
	\end{equation*}
	with $\psi_u\in H^1(\Omega(u))$ denoting the electrostatic potential given as the variational solution to the transmission problem
	\begin{subequations}\label{TMP}
	\begin{align}
	\mathrm{div}(\sigma\nabla\psi_u)&=0 \quad\text{in }\ \Omega(u)\,, \label{TMP1}\\
	\llbracket \psi_u \rrbracket = \llbracket \sigma \partial_z \psi_u \rrbracket &=0\quad\text{on }\ \Sigma(u)\,, \label{TMP2}\\
	\psi_u&=h_u\quad\text{on }\ \partial\Omega(u)\,, \label{TMP3}
	\end{align} 
	with $\sigma =\sigma_1$ in $\Omega_1$ and $\sigma=\sigma_2$ in $\Omega_2(u)$. In \eqref{TMP2}, $\llbracket \cdot \rrbracket$ denotes the jump across the interface $\Sigma(u)$, while \eqref{TMP3} indicates that $\psi_u$ satisfies non-homogeneous Dirichlet boundary conditions prescribed by a given function $h_u$ with an explicit dependence upon the deformation $u$, see \eqref{bobbybrown} below. The latter is such that $h_u\equiv 0$ along the bottom plate $D\times \{-H-d\}$ and $h_u\equiv V$ along the elastic top plate $\mathfrak{G}(u)$ with positive potential value $V>0$ (see assumption~\eqref{bobbybrown1} below).
\end{subequations}
\end{itemize}
The modeling assumption is then that the evolution of $u$ is governed (at least formally) by the gradient flow associated with $E$, which reads
\begin{equation}
\partial_t u+\partial_u E(u)=0\,,\quad t>0\,, \quad u(0)=u_0\,, \label{GF}
\end{equation}
supplemented with clamped boundary conditions; that is,
\begin{equation*}
u(t) \in H_D^2(D):= \left\{ v\in H^2(D) \ :\ v(\pm L) = \partial_x v(\pm L) = 0 \right\}\,, \qquad t\ge 0\,.
\end{equation*}
As already noted in \cite[Section~5]{LW19}, where we studied the existence of minimizers of $E$, the interpretation of equation~\eqref{GF} needs some care for several reasons: 

First, the contact energy involves a non-smooth convex function and it is rather the notion of subdifferential which is appropriate and requires a suitable functional setting. Specifically, we define
\begin{equation*}
S_0 := \left\{ v\in H_D^2(D) \ :\ v>-H \;\text{ in }\; D \right\}\,, \qquad \bar{S}_0 := \left\{ v\in H_D^2(D) \ :\ v\ge -H \;\text{ in }\; D \right\}\,,
\end{equation*}
and recall that $\bar{S}_0$ is a closed convex set in $H_D^2(D)$. The ``derivative'' of $E_c$ with respect to $u$ is then given by the subdifferential $\partial\mathbb{I}_{\bar{S}_0}(u)$ of the indicator function $\mathbb{I}_{\bar{S}_0}$ of the set $\bar{S}_0$. It is a subset of the dual space
\begin{equation*}
H^{-2}(D):=\big(H_D^2(D)\big)'
\end{equation*}
of $H_D^2(D)$ and, for $v\in \bar{S}_0$, it is given by:
\begin{equation*}
\xi \in \partial\mathbb{I}_{\bar{S}_0}(v) \iff \langle \xi , w-v \rangle_{H_D^2} \le 0\,, \qquad w\in \bar{S}_0\,,
\end{equation*}
where $\langle \cdot , \cdot \rangle_{H_D^2}$ denotes the duality pairing between $H^{-2}(D)$ and $H_D^2(D)$. 

Second, the electrostatic energy $E_e(u)$ depends on $u$ not only through the integral over $\Omega(u)$ but also through the solution $\psi_u$ to the transmission problem \eqref{TMP}. Its differentiability is then a tricky and by no means obvious issue but can be handled with the help of shape optimization tools. It follows from the analysis performed in \cite{LW19} that the functional $E_e$ at $u\in \bar{S}_0$ has a directional derivative $g(u)$ given by
\begin{equation}\label{gg}
g(u)(x):=\left\{
\begin{array}{ll}  \displaystyle{\frac{\sigma_2}{2} \big(1+(\partial_x u(x))^2\big) \big(\partial_z\psi_{u,2}(x,u(x))\big)^2} \,, & x\in D\setminus \mathcal{C}(u)\,,\\
\hphantom{x}\vspace{-3.5mm}\\
\displaystyle{\frac{\sigma_1(x,-H)^2}{2\sigma_2}\big(\partial_z\psi_{u,1}(x,-H)\big)^2}\,, & x\in  \mathcal{C}(u)\,,
\end{array}\right.
\end{equation}
where $\psi_{u,1}:= \psi_u|_{\Omega_1}$ and $\psi_{u,1}:= \psi_u|_{\Omega_2(u)}$. In fact, $g(u)$ corresponds to the electrostatic force acting on the elastic plate. It is worth emphasizing here that the derivation of this result owes much to the book by Henrot \& Pierre \cite{HP18} (and its french version \cite{HP05}), which has been a constant source of inspiration in our studies of differentiability properties of the electrostatic energy involved in the modeling of MEMS. In fact, for $u\in S_0$, the coincidence set $\mathcal{C}(u)$ defined in \eqref{coincidence} is empty and the functional $E_e$ is actually Fr\'echet differentiable  at $u$, a feature which is proved along the lines of \cite[Sections~5.3.3-5.3.4]{HP18}, see \cite[Proposition~4.2]{LW19}. The formula \eqref{gg} for $g(u)$ then reduces to its first line. The situation is strikingly different when the coincidence set $\mathcal{C}(u)$ is non-empty, which corresponds to $u\in \bar{S}_0\setminus S_0$. In that situation, the trace of the solution $\psi_u$ to \eqref{TMP} at a point $(x,-H)$, $x\in D$, is given either by the transmission condition \eqref{TMP2} (if $x\in D\setminus\mathcal{C}(u)$) or by the Dirichlet boundary condition \eqref{TMP3} ( if $x\in\mathcal{C}(u)$) and both cases may alternate infinitely often while $x$ ranges in $D$. Identifying the derivative of $E_e$ at such a function $u$ requires a rather delicate analysis, see \cite[Corollary~4.3]{LW19} and Lemma~\ref{L3} below for a precise statement. Let us finally point out that the functional $g(u)$ involves the traces of $\partial_z \psi_u$ on $\mathfrak{G}(u)$ and $\mathcal{C}(u)\times\{-H\}$, which are well-defined only if $\psi_u$ is sufficiently regular. However, since $\Omega(u)$ is only a Lipschitz domain while $\Omega_2(u)$ might be even non-Lipschitz when $\mathcal{C}(u)\ne\emptyset$, the $H^2(\Omega_1)$-regularity of $\psi_{u,1}$ and the $H^2(\Omega_2(u))$-regularity of $\psi_{u,2}$ are not straightforward and a large part of the analysis performed in \cite{LW19} is devoted to this regularity issue. 

Since the computation of the derivative of the mechanical energy with respect to $u$ is classical, collecting the outcome of the above discussion yields the following parabolic variational inequality for $u$:
\begin{subequations} \label{evoleq}
\begin{equation}
\partial_t u + \beta\partial_x^4 u - (\tau+a\|\partial_x u\|_{2}^2) \partial_x^2 u + \partial\mathbb{I}_{\bar S_0}(u) \ni -g(u) \;\text{ in }\; (0,\infty)\times D\,, \label{evoleq1} 
\end{equation}
supplemented with the constraint
\begin{equation}
u(t)\in \bar S_0\,, \qquad t\ge 0\,, \label{evoleq2}
\end{equation} 
and the initial condition
\begin{equation}
u(0)=u_0\,, \qquad x\in D\,. \label{evoleq3}
\end{equation} 
\end{subequations} 
Assuming $\beta>0$, we note that \eqref{evoleq1} is a fourth-order parabolic variational inequality and the main purpose of this paper is to investigate the existence of weak solutions to \eqref{evoleq} for a suitable class of boundary data $h_u$ occurring in \eqref{TMP3}, see \eqref{bobbybrown} below. In the following, we interpret $\partial_x^4 v$ for $v\in \bar{S}_0$ as an element of $H^{-2}(D)$ by virtue of
$$
\langle \partial_x^4 v,\phi\rangle_{H_D^2}:=\int_D \partial_x^2 v \partial_x^2\phi\,\rd x\,,\quad \phi\in H_D^2(D)\,.
$$
A definition of a weak solution to \eqref{evoleq} is then as follows.

\begin{definition}\label{D1}
Let $\beta>0$ and $u_0\in \bar{S}_0$. A weak solution to \eqref{evoleq} is a function 
\begin{equation*}
u\in C([0,\infty), H^1(D))\cap L_{\infty,loc}([0,\infty), H_D^2(D))\cap H_{loc}^1([0,\infty), L_2(D))
\end{equation*}
satisfying \eqref{evoleq2}, \eqref{evoleq3}, the weak formulation
\begin{equation*}
\begin{split}
\int_D \big(  u(t)-u_0\big) v\,\rd x&= - \int_0^t\int_D\left\{g\big(u(s) \big) -(\tau+a\|\partial_x u(s) \|_{2}^2)\partial_x^2 u(s) \right\} v\, \rd x\rd s\\
&\quad\, - \int_0^t\int_D \beta\partial_x^2 u(s) \partial_x^2 v\,\rd x\rd s -\int_0^t\,\langle \zeta(s) ,v\rangle_{H_D^2(D)}\,\rd s
\end{split}
\end{equation*}
for any $t>0$ and $v\in H_D^2(D)$, where 
\begin{equation}
\zeta := -\partial_t u-\beta\partial_x^4 u+\big(\tau+a\|\partial_x u\|_{2}^2\big)\partial_x^2 u - g(u) \in L_{2,loc}([0,\infty),H^{-2}(D)) \label{zeta}
\end{equation} 
with
\begin{equation}\label{constraint}
\zeta(t)\in \partial\mathbb{I}_{\bar S_0} \big(u(t)\big)\ \text{ for a.a. }\ t\ge 0\,,
\end{equation}
and the energy inequality
\begin{equation}\label{400}
\begin{split}
\frac{1}{2}\int_0^t \| \partial_t u(s)\|_{2}^2\,\rd s +E\big(u(t)\big) &\le E\big(u_0\big)\,,\quad t>0 \,.
\end{split}
\end{equation}
\end{definition}

The main result of this paper is then the following existence result.

\begin{theorem}\label{T1}
Let $\beta>0$ and $\tau, a\ge 0$. Suppose that the functions $\sigma$ and $h$ satisfy, respectively, \eqref{S} and \eqref{bobbybrown} below. Then, given $u_0\in \bar{S}_0$, there is a weak solution to \eqref{evoleq} in the sense of Definition~\ref{D1}.
\end{theorem}

Owing to the variational structure \eqref{GF} of \eqref{evoleq}, the proof of Theorem~\ref{T1} is performed with the help of a time implicit Euler scheme. Given a time step $\delta>0$, we construct by induction a sequence $(u_n^\delta)_{n\ge 0}$ such that $u_0^\delta := u_0$ and, for $n\ge 0$, $u_{n+1}^\delta$ is a minimizer of the auxiliary functional
\begin{equation*}
F_{n}^{\delta}(v) := \frac{1}{2\delta} \|v-u_n^\delta\|^2 + E(v)\,, \qquad v\in \bar{S}_0\,.
\end{equation*}
Since $F_{n}^\delta$ includes a negative contribution from the electrostatic energy $E_e$, we begin the proof by showing that $F_n^\delta$ is bounded below, provided $\delta\in (0,\delta_0)$ is sufficiently small, the smallness condition depending only on $D$, the parameters $H$, $d$,$\beta$, $\tau$, $a$, the permittivity $\sigma$, and the function $h$ defining the boundary data in \eqref{TMP3}. The existence of a minimizer $u_{n+1}^\delta$ of $F_n^\delta$ on $\bar{S}_0$ then relies on the lower semicontinuity of the convex part $E_m+E_c$ of the energy and the properties of the electrostatic energy $E_e$ established in \cite{LW19}. As a consequence of $u_{n+1}^\delta$ being a minimizer of $F_n^\delta$ on $\bar{S}_0$, we further derive a handful of estimates on $(u_n^\delta)_{n\ge 1}$, which allows us to show that the family $(u^\delta)_{\delta\in (0,\delta_0)}$ of piecewise constant functions in time defined by
\begin{equation*}
u^\delta(t) := \sum_{n=0}^\infty u_n^\delta \mathbf{1}_{[n\delta,(n+1)\delta)}(t)\,, \qquad t\ge 0\,,
\end{equation*}
has the right compactness properties, so that its cluster points as $\delta\to 0$ are weak solutions to \eqref{evoleq} in the sense of Definition~\ref{D1}. Here again, a key ingredient in the proof is the continuity of the functional $g$ defined in \eqref{gg}, which we established in \cite{LW19}, see Lemma~\ref{L2} below. 

Finally, we address the regularity of the distribution $\zeta\in \partial\mathbb{I}_{\bar{S}_0}(u)$ associated with a weak solution $u$ to \eqref{evoleq} and given by \eqref{zeta}-\eqref{constraint}. As already mentioned, since $\beta>0$, equation~\eqref{evoleq} is a fourth-order parabolic variational inequality and, as such, the regularity of $\zeta(t)$ stemming from \eqref{constraint} is that it is a distribution in $H^{-2}(D)$ for a.e. $t>0$. In fact, since the seminal work \cite{Br72}, regularity for the obstacle problem for the biharmonic parabolic equation has received less attention than the same issue for the obstacle problem for second-order parabolic equations. The only regularity result regarding the obstacle problem for the biharmonic parabolic equation we are aware of is \cite{NO2015}, whereas \cite{CF1979, Fr1973, PL2008, Sc1986}  are devoted to the elliptic analogue. As in \cite{NO2015}, we can prove that $-\zeta(t)$ is actually a non-negative bounded Radon measure on $D$ for a.e. $t>0$.

\begin{corollary}\label{C2}
Let the assumptions of Theorem~\ref{T1} be satisfied. If $u$ is a weak solution to \eqref{evoleq} in the sense of Definition~\ref{D1}, then $u\in L_{2,loc}([0,\infty),H^s(D))$ for $s\in (2,7/2)$ and $-\zeta\in L_{2,loc}([0,\infty),\mathcal{M}_+(D))$, where $\mathcal{M}_+(D)$ is the positive cone of the space $\mathcal{M}(D)=C_0(D)'$ of bounded Radon measures on $D$.
\end{corollary}

Let us finally describe the contents of this paper: in the next section, we state the assumptions on the permittivity $\sigma$ and the boundary data  $h_u$ in \eqref{TMP3}. In Section~\ref{S3}, we recall the well-posedness of the transmission problem \eqref{TMP} and the regularity of its solution established in \cite{LW19}, along with the properties of $E_e$ and $g$ from \cite{LW19} which are needed for the analysis performed below. We also show in this section the existence of a minimizer $u_{n+1}^\delta$ of the functional $F_n^\delta$ on $\bar{S}_0$ for sufficiently small values of the time step $\delta>0$. After this preparation, we are in a position to prove Theorem~\ref{T1} in Section~\ref{S4} and Corollary~\ref{C2} in Section~\ref{S5}.

\section{Assumptions}\label{S2}

We provide now the detailed assumptions we put on the permittivity $\sigma$ and the boundary data $h_u$ occurring in the transmission problem~\eqref{TMP}. As already mentioned, the dielectric properties of the device are accounted for by the permittivity $\sigma$, which is defined by
\begin{equation*}
\sigma(x,z) := \left\{
\begin{array}{lcl}
\sigma_1(x,z) & \text{ for } & (x,z)\in \Omega_1\,, \\
& & \vspace{-4mm}\\
\sigma_2 & \text{ for } & (x,z)\in D\times (-H,\infty)\,,
\end{array}
\right.
\end{equation*}
where 
\begin{equation}\label{S}
\sigma_1\in C^2\big( \overline{\Omega}_1 \big) \ \text{with }\ \sigma_1>0 \text{ in $\overline{\Omega}_1$}\,,\qquad \sigma_2\in (0,\infty)\,.
\end{equation}
In particular, there are $0<\sigma_{min} < \sigma_{max}$ such that
\begin{equation}
\sigma_{min} \le \sigma(x,z) \le \sigma_{max}\,, \qquad (x,z)\in \bar D\times [-H,\infty)\,. \label{sigma}
\end{equation}
We fix $C^2$-functions
\begin{subequations}\label{bobbybrown}
\begin{equation}\label{bobbybrown2a}
h_1: \bar{D}\times [-H-d,-H]\times [-H,\infty)\rightarrow [0,\infty)
\end{equation}
and 
\begin{equation}\label{bobbybrown2aa}
h_2: \bar{D}\times [-H,\infty)\times [-H,\infty)\rightarrow [0,\infty)
\end{equation}
satisfying
\begin{align}
h_1(x,-H,w)&=h_2(x,-H,w)\,,\quad (x,w)\in D\times [-H,\infty)\,,\label{bobbybrown2}\\
 \sigma_1(x,-H)\partial_z h_1(x,-H,w)& =\sigma_2\partial_z h_2(x,-H,w)\,,\quad (x,w)\in D\times [-H,\infty)\,.\label{bobbybrown3}
\end{align}
Moreover, we assume that
\begin{equation}\label{bobbybrown1}
h_1(x,-H-d,w)=h_2(x,w,w)- V=0\,,\quad (x,w)\in \bar{D}\times [-H,\infty)\,,
\end{equation}
where $V>0$, and that there are constants $m_i>0$, $i=1,2,3$, such that 
\begin{equation}
\vert \partial_x h_1(x,z,w)\vert +\vert\partial_z h_1(x,z,w)\vert  \le \sqrt{m_1+m_2 w^2}\,, \quad \vert\partial_w h_1(x,z,w)\vert \le \sqrt{m_3}\,, \label{bobbybrown5}
\end{equation}
for $(x,z,w)\in \bar D \times [-H-d,-H] \times [-H,\infty)$ and
\begin{equation} 
\vert \partial_x h_2(x,z,w)\vert +\vert\partial_z h_2(x,z,w)\vert \le \sqrt{\frac{m_1+m_2 w^2}{H+w}}\,,\quad \vert\partial_w h_2(x,z,w)\vert \le \sqrt{\frac{m_3}{H+w}}\,,\label{bobbybrown6}
\end{equation}
\end{subequations}
for $(x,z,w)\in \bar D \times [-H,\infty) \times [-H,\infty)$.

A typical example for a function $h$ satisfying the assumptions \eqref{bobbybrown} above was given in \cite[Example~5.5]{LW19} which we recall now.

\begin{example}\label{Ex1}
Let us consider the situation where $\sigma_1$ does not depend on the vertical variable $z$; that is, $\sigma_1=\sigma_1(x)$. In that case, we set
$$
h_1(x,z,w):=V\frac{\sigma_2 (H+z+d)}{\sigma_2 d+\sigma_1(x)(H+w)}\,,\quad (x,z,w)\in \bar{D}\times [-H-d,-H]\times [-H,\infty)\,,
$$
and
$$
h_2(x,z,w):=V\frac{\sigma_2 d+\sigma_1(x)(H+z)}{\sigma_2 d+\sigma_1(x)(H+w)}\,,\quad (x,z,w)\in \bar{D}\times [-H,\infty)\times [-H,\infty)\,.
$$
Then assumptions \eqref{bobbybrown} are easily checked. 
\end{example}

For a given function $v\in \bar S_0$ we then define 
\begin{equation}\label{bb}
h_v(x,z):=\left\{\begin{array}{ll} h_{v,1}(x,z):= h_1(x,z,v(x))\,, & (x,z)\in \overline\Omega_1\,,\\
h_{v,2}(x,z):=h_2(x,z,v(x))\,, & (x,z)\in \bar{D}\times [-H,\infty)\,.
\end{array}\right.  
\end{equation}

Let us point out that assumption~\eqref{bobbybrown2}-\eqref{bobbybrown3} guarantee that $h_v$ defined in \eqref{bb} satisfies the transmission conditions \eqref{TMP2}, that is, 
$$
\llbracket h_v \rrbracket = \llbracket \sigma \partial_z h_v \rrbracket =0\quad\text{on}\ \Sigma(v)\,,
$$
while assumption \eqref{bobbybrown1} along with \eqref{TMP3} entails that the electrostatic potential $\psi_v$ equals zero on the bottom plate $D\times \{-H-d\}$ and equals $V$ along the elastic plate $\mathfrak{G}(u)$. Assumptions \eqref{bobbybrown5}-\eqref{bobbybrown6} are required to guarantee the coercivity of the  total energy $E(v)$.

Throughout the paper, $c$ and $(c_i)_{i\ge 1}$ denote positive constants depending only on $D$, $H$, $d$, $\beta$, $\tau$, $a$, $\sigma$, $(m_i)_{1\le i \le 3}$, and $u_0$. The dependence upon additional parameters will be indicated explicitly.

\section{Auxiliary Results}\label{S3}

We first recall some results that were derived in \cite{LW19} and begin with the well-posedness of \eqref{TMP}.

\begin{lemma} \cite[Theorem~1.1]{LW19} \label{L6}
Suppose \eqref{bobbybrown}. For each $v\in \bar{S}_0$, there is a unique variational solution $\psi_v \in h_v + H^1_0(\Omega(v))$ to \eqref{TMP}.  Moreover, $\psi_{v,1}:= \psi_{v}\vert_{\Omega_1} \in H^2(\Omega_1)$, $\psi_{v,2} := \psi_{v}\vert_{\Omega_2(v)} \in H^2(\Omega_2(v))$, and $\psi_{v}$  is a strong solution to the transmission problem~\eqref{TMP} satisfying $\sigma\partial_z \psi_v\in H^1(\Omega(v))$.
\end{lemma}

The regularity of $\psi_v$ stated in Lemma~\ref{L6} guarantees that $g(v)$ defined in \eqref{gg} is meaningful for $v\in \bar{S}_0$. We collect in the next result some properties of $g$ established in \cite{LW19}.
 
\begin{lemma}\label{L2}
Suppose \eqref{bobbybrown}. 
\begin{itemize}
\item[{\bf (a)}] If $v\in \bar{S}_0$ and $(v_j)_{j\ge 1}\subset\bar{S}_0$ are such that $v_j\rightharpoonup v$ in $H^2(D)$, then 
\begin{equation*}
\lim_{j\to\infty} \|g(v_j)-g(v)\|_2 = 0 \quad\text{ and }\quad \lim_{j\to\infty} E_e(v_j) = E_e(v)\,.
\end{equation*}
\item[{\bf (b)}] The mapping $g:\bar{S}_0\rightarrow L_2(D)$ is continuous and bounded on bounded sets, the set $\bar{S}_0$ being endowed with the topology of $H^2(D)$.
\end{itemize}
\end{lemma}

\begin{proof}
Since weak convergence in $H^2(D)$ implies boundedness in $H^2(D)$ and strong convergence in $H^1(D)$, part~{\bf (a)} is shown in \cite[Proposition~3.17 \& Corollary~3.12]{LW19}.
As for part {\bf (b)}, the continuity of $g: \bar{S}_0\rightarrow L_2(D)$ follows from \cite[Theorem~1.4]{LW19}, while the boundedness of $g$ on bounded sets is a consequence of \cite[Corollary~3.14 \& Lemma~3.16]{LW19} and the continuity of the trace from $H^1(\Omega_1)$ to $L_p(D\times\{-H\})$ for all $p\in [1,\infty)$.
\end{proof}

We next turn to differentiability properties of $E_e$. As observed in \cite{LW19}, $E_e$ need not be Fr\'echet differentiable for all $u\in \bar{S}_0$ but it has always directional derivatives.

\begin{lemma} \cite[Proposition~5.6]{LW19} \label{L3}
Suppose \eqref{bobbybrown}. Let $ v\in\bar{S}_0$ and $w\in S_0$. Then
\begin{equation*}
\lim_{s\rightarrow 0^+} \frac{1}{s}\big(E_e( v+s(w-v))-E_e(v)\big)= \int_D g( v) (w-v)\, \rd x\,.
\end{equation*}
\end{lemma}

We also derive a lower bound on $E$.

\begin{lemma}\label{L1}
Suppose \eqref{bobbybrown}. There is a constant $c_1>0$ such that
$$
E(v) \ge \frac{\beta}{4} \|\partial_x^2 v\|_2^2 - c_1 \left( 1 + \|v\|_2^2 \right)\,,\qquad v\in\bar{S}_0\,.
$$
\end{lemma}

\begin{proof}
For the sake  of completeness, we recall the proof performed in \cite[Theorem~5.1]{LW19}. Since $\psi_v$ is a variational solution to \eqref{TMP}, it follows from \eqref{sigma}, \eqref{bobbybrown}, and Young's inequality that
\begin{equation*}
\begin{split}
-E_e(v)&=\frac{1}{2}\int_{\Omega(v)}\sigma\vert\nabla\psi_v\vert^2\,\rd(x,z)\le \frac{1}{2} \int_{\Omega(v)}\sigma\vert\nabla h_v\vert^2\,\rd(x,z)\\
& \le \int_{\Omega(v)} \sigma\left[\left(\partial_x h(x,z,v(x))\right)^2+\left(\partial_w h(x,z,v(x))\right)^2 (\partial_x v(x))^2\right]\,\rd (x,z)\\
&\qquad + \frac{1}{2} \int_{\Omega(v)} \sigma \left(\partial_z h(x,z,v(x))\right)^2\,\rd (x,z)\\
& \le (d+1) \sigma_{max} \int_D \left[ \frac{3}{2} (m_1+m_2v(x)^2) + m_3(\partial_x v(x))^2\right]\,\rd x\,.
\end{split}
\end{equation*}
We next use Poincar\'e's inequality 
\begin{equation}
\|w\|_2 \le |D| \|\partial_x w\|_2\ , \qquad w\in H_0^1(D)\,, \label{poincare}
\end{equation}
and the interpolation inequality
\begin{equation}
\|\partial_x w \|_2^2 \le \|w\|_2 \|\partial_x^2 w\|_2\,, \qquad w\in H^2(D)\cap H^1_0(D)\,, \label{interpolation}
\end{equation}
to obtain
\begin{equation*}
-E_e(v) \le c \left( 1 + \|\partial_x v\|_{2}^2 \right) \le c \left( 1 + \|v\|_2 \|\partial_x^2 v\|_2 \right)\,.
\end{equation*}
Consequently, Young's inequality and the above upper bound yield
\begin{equation*}
E(v) \ge \frac{\beta}{2} \|\partial_x^2 v\|_2^2 - c - c \|v\|_2 \|\partial_x^2 v\|_2 \ge \frac{\beta}{4} \|\partial_x^2 v\|_2^2 - c - c \|v\|_2^2\,,
\end{equation*}
and the proof is complete.
\end{proof}

We next provide the basis for the time implicit scheme used later in order to construct a solution to \eqref{evoleq}.

\begin{lemma}\label{L4}
Set $\delta_0:= \min\{ 1 , (16c_1)^{-1}\}>0$. Then, for any $\delta\in (0,\delta_0)$ and $f\in \bar{S}_0$, there is $v\in \bar{S}_0$ such that
$$
-\frac{1}{\delta}(v-f)-\beta\partial_x^4 v+ (\tau+a\|\partial_x v\|_{2}^2)\partial_x^2 v-g(v)\in\partial\mathbb{I}_{\bar S_0}(v) \,.
$$
Moreover,
$$
\frac{1}{2\delta}\|v-f\|_{ 2}^2+E(v)\le E(f)\,.
$$
\end{lemma}

\begin{proof}
The proof relies on the direct method of calculus of variations. Consider $\delta\in (0,\delta_0)$ and $f\in \bar S_0$ and define 
$$
\mathcal{F}(v) := \frac{1}{2\delta} \|v-f\|_2^2 + E(v)\,,\quad v\in \bar S_0\, .
$$
Then, by Lemma~\ref{L1} and Young's inequality,
\begin{equation*}
\begin{split}
\mathcal{F}(v) & \ge \frac{1}{2\delta} \left( \frac{\|v\|_2^2}{2} - \|f\|_2^2 \right) + \frac{\beta}{4} \|\partial_x^2 v\|_2^2 - c_1 \left( 1 + \|v\|_2^2 \right) \\
& \ge \frac{\beta}{4} \|\partial_x^2 v\|_2^2 + \left( \frac{1}{4\delta} - c_1 \right) \|v\|_2^2 - c_1 - \frac{\|f\|_{2}^2}{2\delta} \\
& \ge \frac{\beta}{4} \|\partial_x^2 v\|_2^2 + 3 c_1 \|v\|_2^2 - c_1 - \frac{\|f\|_{2}^2}{ 2\delta} 
\end{split}
\end{equation*}
for all $v\in \bar{S}_0$. Thus, $\mathcal{F}$ is bounded from below on $\bar{S}_0$ and there is a minimizing sequence $(v_j)_{j\ge 1}$ in $\bar S_0$ satisfying 
\begin{equation}
\mu := \inf_{v\in\bar S_0}\{\mathcal{F}(v)\} \le \mathcal{F}(v_j) \le \mu + \frac{1}{j}\, , \qquad j\ge 1\,. \label{EvEq016}
\end{equation}
 Moreover, the previous lower bound on $\mathcal{F}$ guarantees that $(v_j)_{j\ge 1}$ is bounded in $H_D^2(D)$. Therefore, there is $v\in H_D^2(D)$ such that  (up to a subsequence) 
\begin{subequations}\label{EV}
\begin{align}
v_j & \rightharpoonup v \;\text{ in }\; H^2(D)\, , \label{EvEq017} \\
v_j & \longrightarrow v \;\text{ in }\; H_0^1(D) \,. \label{EvEq018}
\end{align}
\end{subequations}
Clearly, \eqref{EvEq017} ensures that $v\in \bar S_0$ since $\bar S_0$ is closed and convex in $H_D^2(D)$. It then follows from \eqref{EvEq017} and Lemma~\ref{L2}~{\bf (a)}  that
$$
 E_e(v)=\lim_{j\rightarrow \infty} E_e(v_j)\,,
$$
while \eqref{EV} and the weak lower semicontinuity of the $L_2$-norm readily imply that
\begin{equation*}
E_m(v) \le \liminf_{j\to\infty} E_m(v_j)\,.
\end{equation*}
Consequently, 
$$
\mathcal{F}(v)\le \liminf_{j\rightarrow \infty} \mathcal{F}(v_j)=\mu\,,
$$
and we conclude that $v\in\bar S_0$ is a minimizer of $\mathcal{F}$ on $\bar S_0$. This property, in turn, guarantees that, for $w\in  S_0$,
$$
0\le \liminf_{s\rightarrow 0^+} \frac{1}{ s}\big(\mathcal{F}(v+s(w-v))-\mathcal{F}(v)\big)\,.
$$
It then follows from Lemma~\ref{L3} that
\begin{align*}
0 & \le  \int_D \left\{\frac{v-f}{\delta}(w-v)+\beta\partial_x^2 v\,\partial_x^2 (w-v)- \left( \tau+a\|\partial_x v\|_{2}^2 \right) \partial_x v \partial_x (w-v)\right\}\,\rd x \\
& \qquad + \int_D g(v) (w-v)\, \rd x
\end{align*}
for all $w\in S_0$. Since $S_0$ is dense in $\bar{S_0}$,  this inequality also holds for any $w\in \bar{S_0}$. Therefore,
$$
-\frac{1}{\delta}(v-f)-\beta\partial_x^4 v+ (\tau+a\|\partial_x v\|_{2}^2)\partial_x^2 v-g(v)\in\partial\mathbb{I}_{\bar S_0}(v) \,.
$$
Finally, since $f\in \bar S_0$, we have $\mathcal{F}(v)\le \mathcal{F}(f)$, which completes the proof.
\end{proof}

\section{Proof of Theorem~\ref{T1}}\label{S4}

Fix $u_0\in \bar S_0$. For $\delta\in (0,\delta_0)$, we set $u_0^\delta:=u_0$ and, using Lemma~\ref{L4}, we construct by induction a sequence $(u_n^\delta,\zeta_n^\delta)_{n \ge 1}$ in $\bar S_0\times H^{-2}(D)$ such that
\begin{equation}\label{e4}
\zeta_{n+1}^\delta := -A_{n+1}^\delta-\beta\partial_x^4 u_{n+1}^\delta+ (\tau+a\|\partial_x u_{n+1}^\delta\|_{2}^2)\partial_x^2 u_{n+1}^\delta-g(u_{n+1}^\delta)\in\partial\mathbb{I}_{\bar S_0}(u_{n+1}^\delta)\,,
\end{equation}
where
$$
A_{n+1}^\delta:=\frac{1}{\delta}(u_{n+1}^\delta-u_{n}^\delta)\in H_D^2(D)\,,
$$
and
\begin{equation}\label{e5}
\frac{1}{2\delta}\|u_{n+1}^\delta-u_{n}^\delta\|_{2}^2+E(u_{n+1}^\delta)\le E(u_{n}^\delta)
\end{equation}
for $n\ge 0$. Let us first note that \eqref{e5} implies
\begin{equation}\label{e5a}
\frac{1}{2\delta}\sum_{j=0}^n\|u_{j+1}^\delta-u_{j}^\delta\|_{2}^2+E\big(u_{n+1}^\delta\big)\le E(u_0)\,, \qquad n\ge 0\,.
\end{equation}
A first consequence of \eqref{e5a} is an $L_2$-estimate on $(u_n^\delta)_{n\ge 1}$, which is adapted from \cite[Lemma~3.2.2]{AGS08}. More precisely, it follows from H\"older's and Young's inequalities that, for $n\ge 0$,
\begin{align*}
\|u_{n+1}^\delta \|_2^2 - \|u_0\|_2^2 & = \sum_{j=0}^n \left( \|u_{j+1}^\delta \|_2^2 - \|u_j^\delta\|_2^2 \right) \\
& = \sum_{j=0}^n \int_D \left( u_{j+1}^\delta - u_j^\delta \right) \left( u_{j+1}^\delta + u_j^\delta \right) \,\rd x \\
& \le \sum_{j=0}^n \| u_{j+1}^\delta - u_j^\delta \|_2 \| u_{j+1}^\delta + u_j^\delta \|_2 \\
& \le \frac{1}{4 c_1 \delta} \sum_{j=0}^n \| u_{j+1}^\delta - u_j^\delta \|_2^2 + c_1 \delta \sum_{j=0}^n \| u_{j+1}^\delta + u_j^\delta \|_2^2 \,.
\end{align*}
We then infer from Lemma~\ref{L1} and \eqref{e5a} that
\begin{align*}
\|u_{n+1}^\delta \|_2^2 - \|u_0\|_2^2& \le \frac{E(u_0) - E(u_{n+1}^\delta)}{2c_1} + 2 c_1 \delta \sum_{j=0}^n \left( \| u_{j+1}^\delta \|_2^2 + \| u_j^\delta \|_2^2 \right) \\
& \le \frac{E(u_0) +c_1 + c_1 \|u_{n+1}^\delta\|_2^2}{2c_1} +  2c_1\delta \|u_0\|_2^2 + 4 c_1 \delta \sum_{j=1}^{n+1} \| u_{j}^\delta \|_2^2 \,.
\end{align*}
Hence,
\begin{equation*}
\|u_{n+1}^\delta \|_2^2 \le (2+ 4c_1\delta) \|u_0\|_2^2 + 1 + \frac{E(u_0)}{c_1} + 8 c_1 \delta \sum_{j=1}^{n+1} \| u_{j}^\delta \|_2^2\,, \qquad n\ge 0\,.
\end{equation*}
Since $\delta\in (0,\delta_0)$, we have $8c_1\delta<1/2<1$ and we are thus in a position to apply a discrete version of Gronwall's lemma, see \cite[Lemma~3.2.4]{AGS08}, to conclude that
\begin{equation}
\|u_{n+1}^\delta \|_2^2 \le \left( 6 \|u_0\|_2^2 + 2 + \frac{2 E(u_0)}{c_1} \right) e^{16c_1(n+1)\delta}\,, \qquad n\ge 0\,. \label{L2estimate}
\end{equation}
We next use again Lemma~\ref{L1} and \eqref{e5a}, along with \eqref{L2estimate}, to obtain that, for $n\ge 0$,
\begin{equation*}
\begin{split}
& \frac{1}{2\delta}\sum_{j=0}^n\|u_{j+1}^\delta-u_{j}^\delta\|_{2}^2 + \frac{\beta}{4} \|\partial_x^2 u_{n+1}^\delta\|_{2}^2 \\
& \qquad \le \frac{1}{2\delta}\sum_{j=0}^n\|u_{j+1}^\delta-u_{j}^\delta\|_{2}^2 + E\left( u_{n+1}^\delta \right) + c_1 \left( 1 + \|u_{n+1}^\delta \|_2^2 \right) \\
& \qquad \le E(u_0) + c_1 \left( 1 + \|u_{n+1}^\delta \|_2^2 \right) \\
& \qquad \le c_2 e^{16c_1 n\delta}\,.
\end{split}
\end{equation*}
Owing to the functional inequalities \eqref{poincare} and \eqref{interpolation}, we end up with 
\begin{equation}\label{e6}
\frac{1}{2\delta}\sum_{j=0}^n\|u_{j+1}^\delta-u_{j}^\delta\|_{2}^2 + \|u_{n+1}^\delta\|_{H^2}^2 \le c_3 e^{16c_1 n\delta}\,,\qquad n\ge 0\,.
\end{equation}
We next define the functions $u^\delta, A^\delta: [0,\infty) \rightarrow H_D^2(D)$ and $\zeta^\delta: [0,\infty)\rightarrow H^{-2}(D)$ by
$$
u^\delta(t):=\sum_{n\ge 0} u_{n}^\delta \mathbf{1}_{[n\delta,(n+1)\delta)}(t)\,,\quad t\ge 0\,,
$$
$$
A^\delta(t):=\sum_{n\ge  1} A_{n}^\delta \mathbf{1}_{[n\delta,(n+1)\delta)}(t) \,,\quad t\ge 0\,,
$$
and
$$
\zeta^\delta(t):=\sum_{n\ge 1}\zeta_{n}^\delta \mathbf{1}_{[n\delta,(n+1)\delta)}(t)\,,\quad t\ge 0\,,
$$
respectively.

\begin{lemma}\label{L5}
There are a sequence $\delta_\ell\rightarrow 0$ and 
$$
u\in C([0,\infty), H^1(D))\cap L_\infty((0,\infty), H_D^2(D))\cap H_{loc}^1([0,\infty), L_2(D))
$$ 
with $u(0)=u_0$ such that
\begin{equation}
u(t) \in \bar S_0\,, \qquad t\ge 0\,, \label{e202d}
\end{equation}
and, for all $t>0$,
\begin{subequations}\label{e202t}
\begin{align}
u^{\delta_\ell}(t) & \rightarrow u(t)\;\text{ in }\;   H^1(D)\,, \label{e202a} \\
u^{\delta_\ell} & \rightharpoonup u\;\text{ in }\;  L_2((0,t), H_D^2(D))\,, \label{e202b}\\
g\big(u^{\delta_\ell}\big) & \rightarrow g(u)\;\text{ in }\;  L_2((0,t), L_2(D))\,, \label{e202bb} \\
A^{\delta_\ell} & \rightharpoonup \partial_t u\;\text{ in }\;  L_2((0,t), L_2(D))\,. \label{e202c}
\end{align}
\end{subequations}
In particular, 
\begin{equation*}
\zeta^{\delta_\ell} \rightharpoonup \zeta \;\text{ in }\; L_2((0,t),H^{-2}(D)) 
\end{equation*}
for any $t>0$ and 
\begin{equation}\label{e210}
\zeta := -\partial_t u-\beta\partial_x^4 u+\big(\tau+a\|\partial_x u\|_{2}^2\big)\partial_x^2 u-g(u)\in L_{2,loc}([0,\infty),H^{-2}(D))\,.
\end{equation}
\end{lemma}

\begin{proof}
Given $0\le t_1<t_2$ there are integers $n_1\le n_2$ such that $t_i\in [n_i \delta,(n_i+1)\delta)$, $i=1,2$. Either $n_1=n_2$ and $u^\delta(t_2) = u^\delta(t_1)$. Or $n_1<n_2$ and we infer from \eqref{e6} and H\"older's inequality that
\begin{align*}
\| u^\delta(t_2)-  u^\delta(t_1) \|_{2} &= \|u^\delta_{n_2}-u^\delta_{n_1} \|_{2}\le \sum_{j=n_1}^{n_2-1} \|u_{j+1}^\delta-u_{j}^\delta \|_{2} \\
& \le \sqrt{n_2-n_1} \left(\sum_{j=n_1}^{n_2-1} \| u_{j+1}^\delta-u_{j}^\delta \|_{2}^2\right)^{1/2} \\
& \le \sqrt{2 c_3} \sqrt{n_2 \delta-n_1\delta}\, e^{8c_1 n_2 \delta} \,.
\end{align*}
Thus,
\begin{equation}
\| u^\delta(t_2)-  u^\delta(t_1) \|_{2} \le \sqrt{ 2 c_3} \sqrt{t_2 -t_1+\delta}\, e^{8c_1 t_2}\,, \qquad 0\le t_1<t_2\,. \label{e201}
\end{equation}
Moreover, for $t>0$, there is an integer $n\ge 0$ such that $t\in [n\delta,(n+1)\delta)$ and, either $n=0$ and $u^\delta(t)=u_0$, or, again by \eqref{e6}, $\|u^\delta(t)\|_{H^2}^2 \le c_3 e^{16c_1 t}$. Consequently,
\begin{equation}\label{e202}
\| u^\delta(t) \|_{H^2} \le c_4 e^{16c_1 t}\,, \qquad t\ge 0\,.
\end{equation}
Since $H^2(D)$ embeds compactly in $L_2(D)$, we infer from  \eqref{e201} and \eqref{e202} that we may apply a variant of the Arzel\`a-Ascoli theorem, see \cite[Proposition~3.3.1]{AGS08}, and a diagonal argument to obtain the existence of a sequence $(\delta_\ell)_{\ell\ge 1}$, $\delta_\ell\rightarrow 0$, and 
$$
u\in C([0,\infty), L_2(D))\cap L_{\infty,loc}([0,\infty), H_D^2(D))
$$ 
such that
\begin{equation}\label{e202aa}
u^{\delta_\ell}(t)\rightarrow u(t)\;\text{ in }\;   L_2(D) \,, \qquad t\ge 0\,,
\end{equation}
and
\begin{equation*} 
u^{\delta_\ell}\rightharpoonup u\;\text{ in }\;  L_2((0,T), H_D^2(D))\,, \qquad T>0\,. 
\end{equation*}
We have thus proved \eqref{e202b}. Next, an interpolation argument, together with \eqref{e202} and \eqref{e202aa}, yields \eqref{e202a} and the stated time continuity of $u$ in $H^1(D)$. Furthermore, combining \eqref{e202a}, \eqref{e202}, and Lemma~\ref{L2} allows one to apply Lebesgue's theorem to deduce \eqref{e202bb}. Also, since $u^\delta(0)=u_0$ and $u^\delta(t)\in \bar S_0$ for $t\ge 0$ and $\delta\in (0,\delta_0)$, we readily deduce from \eqref{e202aa} that $u(0)=u_0$ and $u(t)\in \bar S_0$ for $t\ge 0$. 

Next, for $t>0$, there is $n_\ell\ge 0$ such that $t\in [n_\ell \delta_\ell,(n_\ell+1)\delta_\ell)$ and it follows from \eqref{e6} that
\begin{equation}\label{iii}
\int_0^t \| A^{\delta_\ell}(s)\|_{2}^2\,\rd s \le \int_0^{(n_\ell+1)\delta_\ell} \| A^{\delta_\ell}(s)\|_{2}^2\,\rd s = \frac{1}{\delta_\ell}\sum_{j=0}^{n_\ell}\|u_{j+1}^\delta-u_{j}^\delta\|_{2}^2 \le 2 c_3 e^{16 c_1 t}\,.
\end{equation}
Since 
\begin{equation*}
A^\delta(t,x) = \frac{u^\delta(t,x)-u^\delta(t-\delta,x)}{\delta}\,, \qquad (t,x)\in (\delta,\infty)\times D\,,
\end{equation*}
and $A^\delta(t,x)=0$ for $(t,x)\in (0,\delta)\times D$,
the sequence $(A^{\delta_\ell})_{\ell\ge 1}$ converges to $\partial_t u$ in $\mathcal{D}'((0,\infty)\times D)$ as $\ell\to \infty$, so that the just established boundedness of $(A^{\delta_\ell})_{\ell\ge 1}$ in $L_{2,loc}([0,\infty),L_2(D))$ implies that $\partial_t u \in L_{2,loc}([0,\infty),L_2(D))$ and the convergence \eqref{e202c} (up to a subsequence). The stated convergence of $\left( \zeta^{\delta_\ell} \right)_{\ell\ge 1}$ and the regularity of $\zeta$ are straightforward consequences of the regularity of $u$ and \eqref{e202t}. 
\end{proof}

We next prove the energy inequality \eqref{400}. 

\begin{lemma}\label{L7}
For $t>0$, 
\begin{equation*}
\frac{1}{2} \int_0^t \|\partial_t u(s)\|_2^2 \,\rd s + E(u(t)) \le E(u_0)\,.
\end{equation*}
\end{lemma}

\begin{proof}
Given $t>0$ and $\ell\ge 0$, we pick again the integer $n_\ell$ such that $t\in [n_\ell \delta_\ell,(n_\ell+1)\delta_\ell)$. Then, by \eqref{e5a}, 
\begin{equation}\label{300}
\frac{1}{2\delta_\ell}\sum_{j=0}^{n_\ell}\|u_{j+1}^\delta-u_{j}^\delta\|_{2}^2+E(u^{\delta_\ell}(t))\le E(u_0) \,.
\end{equation}
Owing to Lemma~\ref{L2}~{\bf (a)}, \eqref{e202a}, and \eqref{e202b}, we have
\begin{equation}\label{301}
\lim_{\ell\to\infty} E_e\big(u^{\delta_\ell}(t)\big)= E_e\big(u(t)\big)\,.
\end{equation}
Since $(u^{\delta_\ell}(t))_{\ell\ge 0}$ is bounded in $H^2(D)$ according to \eqref{e202}, we may extract a further subsequence (not relabeled), possibly depending on $t$, such that $\big(u^{\delta_\ell}(t)\big)_{\ell\ge 0}$ converges to  $u(t)$ weakly in $H^2(D)$ and strongly in $H^1(D)$. Hence
$$
E_m\big(u(t)\big)\le \liminf_{\ell\to\infty} E_m\big(u^{\delta_\ell}(t)\big)\,,
$$ 
which gives, together with \eqref{301},
\begin{equation}\label{302}
E\big(u(t)\big)\le \liminf_{\ell\to\infty} E\big(u^{\delta_\ell}(t)\big)\,.
\end{equation}
Moreover, due to \eqref{e202c} and \eqref{iii}, we have
\begin{equation}\label{303}
\frac{1}{2}\int_0^t \| \partial_t u(s)\|_{2}^2\,\rd s \le  \liminf_{\ell\to \infty}\frac{1}{2\delta_\ell} \sum_{j=0}^{n_\ell}\|u_{j+1}^{\delta_\ell}-u_{j}^{\delta_\ell}\|_{2}^2\,.
\end{equation}
Gathering \eqref{300}, \eqref{302}, and \eqref{303} completes the proof.
\end{proof}

\begin{proof}[Proof of Theorem~\ref{T1}]
To finish off the proof, we are left with showing that $u$ solves the variational inequality \eqref{constraint}. To this end, we recall from \eqref{e4} that
\begin{equation}\label{401}
\begin{split}
0 \le &\, \int_D\left\{ -A_{ n}^\delta+ (\tau+a\|\partial_x u_{n}^\delta\|_{2}^2)\partial_x^2 u_{ n}^\delta-g(u_{ n}^\delta)\right\} \big(u_{n}^\delta-v\big)\,\rd x\\
&
-\beta\int_D \partial_x^2 u_{ n}^\delta \partial_x^2 \big(u_{n}^\delta-v)\,\rd x
\end{split}
\end{equation}
for $n\ge 1$, $\delta\in (0,\delta_0)$, and $v\in \bar S_0$. Now, consider a non-negative function $\phi\in C_c([0,\infty))$ and $w\in L_{2,loc}([0,\infty),H^2(D))$ such that $w(t)\in \bar{S}_0$ for a.e. $t\in (0,\infty)$. Then, for $\delta$ small enough, $\mathrm{supp}\ \phi\subset (\delta,\infty)$ and, by \eqref{401},
\begin{align}
& \int_{0}^{\infty} \phi(s) \int_D \left\{ -A^{\delta}(s)+ \left( \tau+a\|\partial_x u^{\delta}(s)\|_{2}^2 \right) \partial_x^2 u^{\delta}(s)-g(u^{\delta}(s))\right\} \big(u^{\delta}(s)-w(s)\big)\,\rd x\rd s \nonumber \\
&\qquad \qquad -\int_{0}^{\infty} \phi(s) \int_D \beta\partial_x^2 u^{\delta}(s) \partial_x^2 \big(u^{\delta}(s)-w(s))\,\rd x\rd s \nonumber \\
& \qquad = \sum_{n=1}^{\infty}\int_{n \delta}^{(n+1) \delta} \phi(s) \int_D\left\{ -A_{n}^{\delta}+ \left( \tau+a\|\partial_x u_{n}^{\delta}\|_{2}^2 \right) \partial_x^2 u_{n}^{\delta}-g(u_{n}^{\delta})\right\} \big(u_{n}^{\delta}-w(s)\big)\,\rd x\rd s \nonumber \\
&\qquad\qquad -\sum_{n=1}^{\infty} \int_{n \delta}^{(n+1) \delta} \phi(s) \int_D \beta \partial_x^2 u_{n}^{\delta} \partial_x^2 \big(u_{n}^{\delta}-w(s))\,\rd x\rd s \nonumber \\
&\qquad \ge 0\,. \label{pam}
\end{align}
On the one hand, it follows from \eqref{e202t} that
\begin{equation*}
\begin{split}
& \lim_{\ell\to\infty} \int_{0}^{\infty} \phi(s) \int_D \left\{ -A^{\delta_\ell}(s)+ (\tau+a\|\partial_x u^{\delta_\ell}(s)\|_{2}^2)\partial_x^2 u^{\delta_\ell}(s)-g(u^{\delta_\ell}(s))\right\} \big(u^{\delta_\ell}(s)-w(s)\big)\,\rd x\rd s \\
& \qquad = \int_{0}^{\infty} \phi(s) \int_D \left\{ -\partial_t u(s)+ (\tau+a\|\partial_x u(s)\|_{2}^2)\partial_x^2 u(s)-g(u(s))\right\} \big(u(s)-w(s)\big)\,\rd x\rd s\,.
\end{split}
\end{equation*}
On the other hand, we infer from \eqref{e202b} and the non-negativity of $\phi$ that
\begin{equation*}
\lim_{\ell\to\infty} \int_{0}^{\infty} \phi(s) \int_D \beta\partial_x^2 u^{\delta_\ell}(s) \partial_x^2 w(s)\,\rd x\rd s = \int_{0}^{\infty} \phi(s) \int_D \beta\partial_x^2 u(s) \partial_x^2 w(s)\,\rd x\rd s
\end{equation*}
and
\begin{equation*}
\liminf_{\ell\to\infty} \int_{0}^{\infty} \phi(s) \int_D \beta|\partial_x^2 u^{\delta_\ell}(s)|^2 \,\rd x\rd s \ge \int_{0}^{\infty} \phi(s) \int_D \beta|\partial_x^2 u(s)|^2 \,\rd x\rd s\,.
\end{equation*}
Collecting the above identities, taking $\delta=\delta_\ell$ in \eqref{pam}, and letting $\ell\to\infty$, we conclude that
\begin{equation*}
\begin{split}
\int_{0}^{\infty} \phi(s) \int_D &\left\{ -\partial_t u(s)+ (\tau+a\|\partial_x u (s)\|_{2}^2)\partial_x^2 u (s)-g(u (s))\right\} \big(u (s)-w(s)\big)\,\rd x\rd s\\
&\qquad \quad -\int_{0}^{\infty} \phi(s) \int_D \beta\partial_x^2 u (s) \partial_x^2 \big(u (s)-w(s))\,\rd x\rd s \ge 0\,.
\end{split}
\end{equation*}
That is, recalling the definition \eqref{e210} of $\zeta$,
$$
\int_{0}^{\infty} \phi(s)\ \langle \zeta(s), u(s)-w(s)\rangle_{H_D^2}\,\rd s \ge 0
$$
for any $w\in L_{2,loc}([0,\infty), H_D^2(D))$ satisfying $w(t)\in \bar S_0$ for a.e. $t\in (0,\infty)$ and any non-negative $\phi\in C_c([0,\infty))$. In particular, for all $v\in \bar{S}_0$ and non-negative $\phi\in C_c([0,\infty))$, the choice $w(t)\equiv v$, $t>0$, in the above inequality gives
\begin{equation*}
\int_{0}^{\infty} \phi(s)\ \langle \zeta(s), u(s)-v\rangle_{H_D^2}\,\rd s \ge 0\,,
\end{equation*}
which implies, since $\bar{S}_0$ is separable, that
$$
\zeta(t)\in \partial\mathbb{I}_{\bar S_0} \big(u(t)\big)\ \text{ for a.a. }\ t\ge 0\,.
$$
Finally, since $\partial_t u\in L_{2,loc}([0,\infty), L_2(D))$, it follows from the definition \eqref{e210} of $\zeta$ that $u$ solves \eqref{evoleq} in the sense of Definition~\ref{D1}, and the proof of  Theorem~\ref{T1} is complete.
\end{proof}

\section{Proof of Corollary~\ref{C2}}\label{S5}

We finally derive the additional features enjoyed by $\zeta$ as stated in Corollary~\ref{C2}.

\begin{proof}[Proof of Corollary~\ref{C2}]
Let $u$ be a weak solution to \eqref{evoleq} in the sense of Definition~\ref{D1} and define $\zeta$ by \eqref{zeta}. We introduce the set 
\begin{equation}
\mathcal{Z} := \{ t\in (0,\infty)\, :\, \zeta(t)\in \partial\mathbb{I}_{\bar S_0} \big(u(t)\big) \}\,, \label{zeta0}
\end{equation}
and observe that $|\mathcal{Z}|=0$ since $\zeta$ satisfies \eqref{constraint}. Moreover, since $u(t)+v$ belongs to $\bar{S}_0$ for any non-negative $v\in C_c^\infty(D)$, it readily follows from \eqref{zeta0} that
\begin{equation}
\langle -\zeta(t) , v \rangle_{H_D^2} \ge 0\,, \qquad t\in \mathcal{Z}\,. \label{zeta00}
\end{equation}
That is, for $t\in \mathcal{Z}$, $-\zeta(t)$ is a non-negative distribution on $D$ and thus a non-negative Radon measure, see, e.g., \cite[Proposition~6.6]{DD2012}.

Next, let $T>0$. According to the regularity of $u$, 
\begin{equation}
K_T := \sup_{t\in [0,T]}\{\|u(t)\|_{H^2}\} < \infty\,. \label{zeta1}
\end{equation}
\noindent\textbf{Step~1.} Let $t\in [0,T]$ and $x\in [0,L]$. Since $u(t)\in H_D^2(D)$, it follows from \eqref{zeta1} that 
\begin{align*}
u(t,x) & = u(t,L) + (x-L) \partial_x u(t,L) + \int_x^L (y-x) \partial_x^2 u(t,y)\,\rd y \\
& \ge - \int_x^L (y-x) |\partial_x^2 u(t,y)|\,\rd y \ge - \frac{(L-x)^{3/2}}{\sqrt{3}} \|\partial_x^2 u(t)\|_{2} \\
& \ge - K_T (L-x)^{3/2}\,.
\end{align*}
Hence, for $x\in [x_T,L]$ with 
 \begin{equation*}
x_T := \max\left\{ L - \left( \frac{H}{2K_T} \right)^{2/3} , \frac{L}{2} \right\}\in [L/2,L)\,, 
\end{equation*}
we deduce that $u(t,x)\ge -H/2$. Using the same argument for $x\in [-L,0]$, we end up with
\begin{equation}
u(t,x) \ge -\frac{H}{2}\,, \qquad x\in [-L,-x_T] \cup [x_T,L]\,, \quad t\in [0,T]\,. \label{zeta2}
\end{equation}
Now, consider $t\in [0,T]\cap \mathcal{Z}$ and $v\in C_c^\infty(D)$ such that $\mathrm{supp}\, v \subset  [-L,-x_T] \cup [x_T,L]$. For $\theta\in (0,1)$ small enough, we infer from \eqref{zeta2} that $u(t) \pm\theta v$ belongs to $\bar{S}_0$, so that \eqref{constraint} entails
\begin{equation*}
0 \le \langle \zeta(t) , u(t) - u(t) \mp \theta v \rangle_{H_D^2} = \mp \theta \langle \zeta(t) , v \rangle_{H_D^2}\,.
\end{equation*}
Since $\theta$ is positive, we conclude that $\langle \zeta(t) , v \rangle_{H_D^2}=0$. Consequently,
\begin{equation}
\mathrm{supp}\, \zeta(t) \subset [-x_T,x_T]\,, \qquad t\in [0,T]\cap \mathcal{Z}\,. \label{zeta3}
\end{equation}

\medskip

\noindent\textbf{Step~2.} We now fix a non-negative function $\chi_T\in C_c^\infty(D)$ such that $\chi_T\equiv 1$ on $[-x_T,x_T]$. Then, for $v\in C_c^\infty(D)$ and $t\in [0,T]\cap \mathcal{Z}$, the function $u(t) + \chi_T (\|v\|_\infty \pm v)$  belongs to $\bar{S}_0$ and it follows from \eqref{zeta00} and \eqref{zeta3} that
\begin{equation*}
0 \le \langle -\zeta(t) , \chi_T (\|v\|_\infty \pm v) \rangle_{H_D^2} =  \|v\|_\infty \langle -\zeta(t) , \chi_T \rangle_{H_D^2} \pm \langle -\zeta(t) , v \rangle_{H_D^2}\,,
\end{equation*}
the identity $\langle \zeta(t) , \chi_T v \rangle_{H_D^2} = \langle \zeta(t) , v \rangle_{H_D^2}$ being guaranteed 
by \eqref{zeta3} and the properties of $\chi_T$. Thus, 
\begin{equation*}
\left| \langle -\zeta(t) , v \rangle_{H_D^2} \right| \le \|v\|_\infty \langle - \zeta(t) , \chi_T \rangle_{H_D^2} \,,
\end{equation*}
and the density of $C_c^\infty(D)$ in $C_0(D)$ and the already established non-negativity of $-\zeta$ ensure that $-\zeta(t)$ belongs to $\mathcal{M}_+(D)$ with
\begin{equation}
\|-\zeta(t)\|_{\mathcal{M}} \le \langle - \zeta(t) , \chi_T \rangle_{H_D^2}\,, \qquad t\in [0,T]\cap \mathcal{Z}\,. \label{zeta4}
\end{equation}
Since $\zeta\in L_2((0,T),H^{-2}(D))$, an immediate consequence of \eqref{zeta4} is that $\zeta\in L_2((0,T),\mathcal{M}(D))$.

Finally, let $s\in (2,7/2)$. According to \cite[Lemma~4.1~(iii)]{AQ2004}, $\mathcal{M}(D)$ is continuously embedded in $H_D^{s-4}(D)$, so that the just established regularity of $\zeta$ implies that $\zeta\in L_2((0,T),H_D^{s-4}(D))$. Together with the regularity of $u$, $\partial_t u$, and $g(u)$, this property and \eqref{zeta} ensure that $\partial_x^4 u$ belongs to $L_2((0,T),H_D^{s-4}(D))$. Consequently, it follows from elliptic regularity theory that  $u$ belongs to $L_2((0,T),H_D^{s}(D))$. This completes the proof of Corollary~\ref{C2}.
\end{proof}

\begin{remark}
Since $u\in C([0,\infty)\times \bar{D})$ by Theorem~\ref{T1} and since $H^1(D)$ embeds continuously in $C(\bar{D})$, it easily follows from \eqref{constraint} (by an argument similar to that leading to \eqref{zeta3}) that 
\begin{equation*}
\mathrm{supp}\, \zeta(t) \subset \mathcal{C}(u(t)) \;\text{ for a.e. }\; t\in (0,\infty)\,,
\end{equation*}
the coincidence set $\mathcal{C}(u(t))$ being defined in \eqref{coincidence}.
\end{remark}

\bibliographystyle{siam}
\bibliography{BG_Transmission_Model}

\end{document}